\documentclass[11pt,a4paper]{article}

\usepackage[T1]{fontenc}
\usepackage[utf8]{inputenc}
\usepackage{authblk}
 \usepackage{amssymb,amsmath,amsthm, latexsym, wasysym}
 \usepackage{enumerate}
 \usepackage{amsfonts}

 \theoremstyle{plain}
 \newtheorem{theorem}{Theorem}[section]
 \newtheorem{lemma}[theorem]{Lemma}
 
 \newtheorem{proposition}[theorem]{Proposition}
 \theoremstyle{definition}
 \newtheorem{definition}[theorem]{Definition}

 \theoremstyle{remark}
 \newtheorem{remark}{Remark}
 
 \numberwithin{equation}{section}
 \usepackage[T1]{fontenc}
 \usepackage{mathrsfs}
 \usepackage{setspace}
 \usepackage{graphicx,subfigure}
 \usepackage{xcolor}
 \usepackage{epstopdf}
 \usepackage[utf8]{inputenc}
 \usepackage{multicol}
 
  \usepackage{amsmath}
   \usepackage{amssymb}
  \usepackage{amsthm}
 \usepackage{upgreek}
 \usepackage{amsfonts}
 \usepackage{contour}
 \usepackage{float}
 \usepackage{breqn}
 \usepackage{textcomp}
 \usepackage{gensymb}
 \usepackage[left=1.3in,right=1in,top=1in,bottom=1in]{geometry}
 \usepackage{hyperref}
 \newcommand{\remove}[1]{}

\newcommand{\al} {\alpha}

\newcommand{\ga} {\gamma}

\newcommand{\Om} {\Omega}
\newcommand{\ra} {\rightarrow}

\newcommand{\la} {\lambda}

\newcommand{\noi} {\noindent}

\newcommand{\RR} {\mathbb{R}}

\newcommand{\DD} {\displaystyle}
\newcommand{\ro} {\rho}
\newcommand{\R}{\mathbb R^N\times\mathbb R^N}

\author[ ]{Reshmi Biswas\thanks{b.reshmi@iitg.ac.in}}
\author[ ]{Sweta Tiwari\thanks{swetatiwari@iitg.ac.in}}
\affil[ ]{Department of Mathematics}
\affil[ ] {Indian Institute of Technology Guwahati\\
Guwahati}\affil [ ] {Assam 781039\\ India.}

\begin{document}
 \title {Variable order nonlocal Choquard problem with variable exponents 
}

\date{}
\maketitle

\begin{abstract}
	In this article, we study the existence/multiplicity results for the following variable order nonlocal 
	Choquard problem with variable exponents
\begin{equation*}
\begin{array}{rl}
(-\Delta)_{p(\cdot)}^{s(\cdot)}u(x)&=\lambda|u(x)|^{\alpha(x)-2}u(x)+\left(\DD\int_\Omega\frac{F(y,u(y))}{|x-y|^{\mu(x,y)}}dy\right)f(x,u(x)),\\
&~\hspace{6cm} x\in \Omega, \\
u(x)&=0 ,\hspace{20mm} x\in \Omega^c:=\mathbb R^N\setminus\Omega,
\end{array}
\end{equation*}
where  $\Om\subset\mathbb R^N$ is a smooth and bounded domain, $N\geq 2$,  $p,s,\mu$ and $\alpha$ are continuous functions on $\mathbb R^N\times\mathbb R^N$ 
and $f(x,t)$ is Carath\'edory function with $F(x,t):=\displaystyle\int_{0}^{t} f(x,s)ds$.
Under suitable assumption on $s,p,\mu,\alpha$ and $f(x,t)$, first we study
the analogous Hardy-Sobolev-Littlewood-type result for variable exponents suitable for the fractional Sobolev space with variable order and
variable exponents.
Then we give the existence/multiplicity results for the above equation.\\\\
\noi {\bf Key words}: Choquard problem, Hardy-Sobolev-Littlewood inequality, Variable order fractional $p(\cdot)$- Laplacian, concave-convex nonlinearities.\\\\
 \noi {2010 Mathematics Subject Classification, Primary: 35J60, 35R11, 46E35.}
\end{abstract}
\maketitle
\newpage

\section{Introduction}
\noi Consider the following nonlocal Choquard problem involving variable order and variable exponents
\begin{equation}\label{mainprob*}
\;\;\;	\left.\begin{array}{rl}
(-\Delta)_{p(\cdot)}^{s(\cdot)} u(x)&=\left(\DD\int_\Om\frac{F(y,u(y))}{|x-y|^{\mu(x,y)}}dy\right)f(x,u(x)),\hspace{3mm}
x\in \Om, \\
u&=0 ,\hspace{30mm} x\in \Om^c:=\mathbb R^N\setminus\Om,
\end{array}
\right\}
\end{equation}
where  $\Om\subset\mathbb R^N$ is a smooth and bounded domain, $N\geq 2$,  $p\in C(\RR^N\times \RR^N,(1,\infty))$, $s\in C(\RR^N\times \RR^N, (0,1))$ 
and $\mu\in C(\RR^N\times \RR^N,\mathbb R)$.
Also $f\in C(\Om\times\RR,\RR)$ is a Carath\'eodory function with the antiderivative $F(x,t)$ given by
$$F(x,t):=\DD\int_{0}^{t} f(x,s)ds.$$
The nonlocal operator  $(-\Delta)_{p(\cdot)}^{s(\cdot)}$ is defined, in the sense of Cauchy principle value, as
\begin{equation}\label{operator}
(-\Delta)_{p(\cdot)}^{s(\cdot)} u(x):=  P.V.
\int_{\RR^N}\frac{|
	u(x)-u(y)|^{p(x,y)-2}(u(x)-u(y))}{|
	x-y|^{N+s(x,y)p(x,y)}}dy, ~~x \in \RR^N,
\end{equation}
The nonlinearity on the right side of \eqref{mainprob*} is motivated by the equation
 $$-\Delta u+Vu=(I_\alpha\ast u^p)u^{p-1},$$
where $I_\alpha:\RR^N\rightarrow\RR$ is a Riesz potential defined for $x\in\RR^N\setminus\{0\}$ by
$$I_\alpha(x)=\frac{A_\alpha}{|x|^{N-\alpha}}.$$
 P. Choquard introduced this equation in for describing an electron trapped in its own hole. 
 This equation also appears in the theory of the polaron at rest (\cite{pekar}) and in modeling self-gravitating matter as given in \cite{moroz5}.
 In \cite{Lieb1}, Lieb has studied the existence and uniqueness of the minimizing solution for the problem
\[-\Delta u+u=(|x|^{\mu}\ast F(u))f(u) \text{ in } \RR^N,\]
where $f(t)$ has critical growth.
The multiplicity results for the Brezis-Nirenberg type problem of the nonlinear Choquard equation has been studied by Gao and Yang in (\cite{Gao-yang}-\cite{gao-yang1}).
Here the author have considered
 \[-\Delta u =\left(\DD\int_\Om\frac{|u|^{2^\ast_\mu}}{|x-y|^{\mu(x,y)}}dy\right)u^{2^\ast_\mu-1} +\la u,\hspace{3mm}
x\in \Om, \\
\text{ and }u=0 ,\;\; x\in \partial\Om,\]
where $2^\ast_{\mu}=\frac{(2N-\mu)}{(N-2)}, 0<\mu<N$. 
For more results on Choquard problem involving concave-convex nonlinearities we refer (\cite{sreenadh3}-\cite{sreenadh2},\cite{wang}-\cite{wang2}).\\
In recent years, problems involving nonlocal
operators have gained a lot of attentions due to their occurrence in real-world applications, such as, the thin obstacle problem,
optimization, finance, phase transitions and also in pure 
mathematical research, such as, minimal surfaces, conservation laws etc.
The celebrated work of Nezza et al. \cite{HG} provides the necessary functional set-up to study these nonlocal
problems using variational method. We refer \cite{Bisci} and references therein for more details on problems involving semi-linear fractional Laplace operator.
In continuation to this, the problems involving
quasilinear nonlocal fractional  $p$-Laplace operator are extensively studied by 
many researchers including
Squassina, Palatucci, Mosconi,  R\u{a}dulescu et al. (see \cite{FrPa, rad1, sq1, MoSq} ), where the authors  studied various aspects, viz., existence, multiplicity
and regularity of the solutions of the quasilinear nonlocal problem involving  fractional $p$-Laplace operator.
Recently, Choquard problem involving nonlocal operators have been studied by Squassina et.al in \cite{squassina} and Mukerjee and Sreenadh in \cite{sreenadh2}.
In \cite{sreenadh2}, authors have discussed the Brezis–Nirenberg type problem of nonlinear Choquard equation involving the fractional Laplacian
\[(-\Delta)^s u =\left(\DD\int_\Om\frac{|u|^{2^\ast_{\mu,s}}}{|x-y|^{\mu(x,y)}}dy\right)u^{2^\ast_{\mu,s}-2}u +\la u,\hspace{3mm}
x\in \Om, \\
\text{ and }u=0 ,\;\; x\in \RR^N\setminus\Om.\]\\
As the variable growth on the exponent $p$ in the local $p(x)$-Laplace operator, defined as div$(| \nabla u| ^{p(x)-2}\nabla u)$, 
makes it more suitable for modeling the problems like image restoration, obstacle problems compared to $p$-Laplace operator,  
henceforth, it is a natural inquisitiveness  to substitute	
the nonlocal fractional $p$-Laplace operator with the nonlocal operator involving variable exponents and variable order as defined in \eqref{operator} and expect 
better modeling. In analogy to the Lebesgue spaces with variable exponents (see \cite{Fan1, Die}), recently
Kaufmann et al.  introduced the fractional Sobolev spaces with variable exponents in \cite{KRV}. Some results involving fractional $p(\cdot)$-Laplace operator 
and associated fractional Sobolev spaces with variable exponents are
studied in \cite{ AB,AR}.\\
Very recently Alves, R\u{a}dulescu and Tavares have studied generalized
Choquard equations driven by non-homogeneous operators in \cite{Alves4}. In \cite{Alves5}, Alves et.al have proved a Hardy-Littlewood-Sobolev-type inequality for
variable exponents and used it to study the quasilinear Choquard equations involving
variable exponents.
Motivated by this, in our present work, we will establish a Hardy-Littlewood-Sobolev-type result for the functions in nonlocal Sobolev spaces with variable order 
and variable exponents
as defined in the section \ref{sec 3}. 
Then using this result, we study the
combined effect of concave and convex nonlinearities on the
existence and multiplicity of solutions for nonlocal Choquard problem involving variable order and variable exponents. 
To the best of our knowledge, this is the first work addressing the variable order nonlocal Choquard problem with variable exponents.\\
In section \ref{sec 2} we first state the main results of this article. In section \ref{sec 3} we give the preliminaries
on fractional Sobolev spaces with variable order and variable exponents. 
We would like to emphasis that here we are defining and establishing the embedding theorems for the nonlocal fractional order spaces,
when the exponent $p$ as well as the order $s$ admit 
variable growth.
We give the proofs of the main results in section \ref{sec 4}.

  \section{Statements of the main theorems}\label{sec 2}
 First for any real valued function $\Phi$ defined on a domain $\mathcal{D}$, we set 
{\small\begin{equation}\label{notation1}
	\Phi^{-}:=\inf_{\mathcal{D}} \Phi(x)\text{ and } \Phi^{+}:=\sup_{ \mathcal{D}}\Phi(x).
	\end{equation}}
\noindent We also define the function space
$$C_+(\mathcal{D}):=\{\Phi\in C(\mathcal{D}, \RR):1 <\Phi^{-}\leq \Phi^{+}<\infty\}.$$
We consider the following assumptions on the variable order $s$, variable exponents $p$ and $\mu$ appearing in \eqref{mainprob*}.
\begin{itemize}
	\item[{\bf(S1)}] $s:\mathbb R^N\times\mathbb R^N\rightarrow\mathbb R$ is a continuous and symmetric function, i.e., $s(x,y)=s(y,x)$ 
	 for all $(x,y)\in \RR^N\times \RR^N$
	with $0<s^-\leq s^+<1$.
	\item[{\bf(P1)}] $p\in C_+(\mathbb R^N\times\mathbb R^N)$ is continuous and symmetric function, i.e., $p(x,y)=p(y,x)$ for all $(x,y)\in \RR^N\times \RR^N$
	 with $s^+p^+<N$.
	\item[{\bf($\mu$1)}]  $\mu:\mathbb R^N\times\mathbb R^N\rightarrow\mathbb R$ is symmetric function, i.e., $\mu(x,y)=\mu(y,x)$
	for all $(x,y)\in \RR^N\times \RR^N$ with $0<\mu^-\leq \mu^+<N$.
	\end{itemize}
\begin{theorem}\label{HLS1}
Let $\Om$ be a smooth bounded domain in $\RR^N,~N\geq2.$ Assume $(S1),(P1)$ and $(\mu1)$ hold and $s$ and $p$ are uniformly continuous in $\RR^N\times\RR^N$.
 Let $q\in C_+(\R)$ such that
 \begin{equation*}
  \label{q1}
  \frac{2}{q(x,y)}+\frac{\mu(x,y)}{N}=2 \text{ for all }(x,y)\in \RR^N\times \RR^N,
 \end{equation*}
 and $r\in C_+(\RR^N)\cap\mathcal{M}$ where
\begin{equation*}
 \label{r1}
 \mathcal{M}:=\left\{p(x,x)\leq r(x)q^-\leq r(x)q^+<p_s^*(x):=\frac{Np(x,x)}{N-s(x,x)p(x,x)}\text{ for all } x\in\RR^N\right\}.
\end{equation*}
Then for $u\in W^{s(x,y),p(x,x),p(x,y)}(\RR^N)$ (as defined in section \ref{sec 3}), $|u|^{r(\cdot)}\in L^{q^+}(\RR^N)\cap L^{q^-}(\RR^N)$ with
\begin{equation*}
 \label{hls1}
 \int_{\RR^N}\int_{\RR^N}\frac{|u(x)|^{r(x)}|u(y)|^{r(y)}}{|x-y|^{\mu(x,y)}}dxdy\leq C\left(\||u|^{r(\cdot)}\|^2_{L^{q^+}(\RR^N)}+\||u|^{r(\cdot)}\|^2_{L^{q^-}(\RR^N)}\right)
\end{equation*}
where $C>0$ is a constant, independent of $u$.
\end{theorem} 
Next using the above inequality, we study the existence of the solution of
the following variable order nonlocal Choquard equation with variable exponents:
\begin{equation}\label{mainprob1}
\;\;\;	\left.\begin{array}{rl}
(-\Delta)_{p(\cdot)}^{s(\cdot)} u(x)&=\left(\DD\int_\Om\frac{F(y,u(y))}{|x-y|^{\mu(x,y)}}dy\right)f(x,u(x)),\hspace{3mm}
x\in \Om, \\
u&=0 ,\hspace{15mm} x\in \mathbb R^N\setminus\Om,
\end{array}
\right\}
\end{equation}
where $f\in C(\Om\times\Om,\RR)$ is a Carath\'eodory function and $F(x,t)=\int_0^t f(x,s)ds$ for a.e. $x\in\Om$ with the following
assumptions:
 \begin{itemize}
  \item[(F1)] There exists a constant $M>0$ and a function $r\in C_+(\RR^N)\cap\mathcal{M}$  with
   $r^->p^+$ such that
  \begin{equation*}
   |f(x,t)|\leq M(|t|^{r(x)-1}), \text{ for all }(x,t)\in \Om\times\RR.
  \end{equation*}
\item[${(F2)}$] There exist a constant  $\Theta>p^+$ such that $0<\Theta F(x,t)\leq 2tf(x,t)$ for all nonzero $t\in\RR$ and for a.e. $x\in\Om$.
 \end{itemize}
We define the weak solution of problem \eqref{mainprob1} in the functional space $X_{0}$, defined in section \ref{sec 3}, as follows:
\begin{definition}\label{defi 2.1}
	A function $u\in X_0$ is called a weak solution of \eqref{mainprob1}, if for every $w\in X_{0}$ we have 
	\begin{align*}
		\int_{{\RR^N}\times{\RR^N}}\frac{|
			u(x)-u(y)|^{p(x,y)-2}(u(x)-u(y))(w(x)-w(y))}{|
			x-y|^{N+s(x,y)p(x,y)}}dxdy\nonumber\\=\frac{1}{2}\int_{\Om}\int_{\Om}\frac{F(x,u(x))f(y,u(y))w(y)}{|x-y|^{\mu(x,y)}}dxdy.
			\end{align*} 
\end{definition}
We have the following existence result for the nontrivial solution for \eqref{mainprob1}.
\begin{theorem}
 \label{existence}
Let $\Om$ be a smooth bounded domain in $\RR^N,~N\geq2.$ Assume $(S1),(P1)$ and $(\mu1)$ and let $q$ be as in \eqref{q1}. 
 Also  assume that $f(x,t)$ satisfies $(F1)$ and $(F2)$.
Then the problem \eqref{mainprob1} admits a nontrivial solution.
\end{theorem}
 Next motivated by the pioneer work of Cerami et al. \cite{ABC} on problems involving concave and convex nonlinearities in case of local operator and \cite{ BERF} in case of nonlocal
operator, we study the existence of multiple solutions for variable order nonlocal problem 
with variable exponents involving concave and convex nonlinearities.           
\begin{equation}\label{mainprob}
{\small \left.\begin{array}{rl}
(-\Delta)_{p(\cdot)}^{s(\cdot)} u(x)&=\lambda| u(x)|^{\al(x)-2}u(x)+\left(\DD\int_\Om\frac{F(y,u(y))}{|x-y|^{\mu(x,y)}}dy\right)f(x,u(x)),\hspace{.5mm}
x\in \Om, \\
u&=0 ,\hspace{15mm} x\in \mathbb R^N\setminus\Om,
\end{array}
\right\}}
\end{equation}
where $f$ and $F$ be as in Theorem \ref{existence}.
\begin{definition}\label{defi 2.2}
	A function $u\in X_0$ is called a weak solution of \eqref{mainprob}, if for every $w\in X_{0}$ we have 
	\begin{align*}
		\int_{{\RR^N}\times{\RR^N}}\frac{|
			u(x)-u(y)|^{p(x,y)-2}(u(x)-u(y))(w(x)-w(y))}{|
			x-y|^{N+s(x,y)p(x,y)}}dxdy\nonumber\\=\la\int_{\Om}| u(x)|^{\al(x)-2}u(x)w(x) dx
			+\frac{1}{2}\int_{\Om}\int_{\Om}\frac{F(x,u(x))f(y,u(y))w(y)}{|x-y|^{\mu(x,y)}}dxdy.
			\end{align*} 
\end{definition}
\begin{theorem}\label{multiplicity}
Let $\Om$ be a smooth bounded domain in $\RR^N,~N\geq2.$ Assume $(S1),(P1)$, $(\mu1)$ and $(F1)-(F2)$ and let $q$ be as in \eqref{q1}. We also assume 
that the variable exponents $\al(\cdot)\in C_+(\overline\Om)$ such that $\al^+<p^-.$
	Then there exists $\Lambda>0 $ such that for all $\la \in(0,\Lambda)$,  the problem $(\ref{mainprob})$ admits at 
	least two distinct nontrivial weak solutions.
\end{theorem}
\section{Fractional Sobolev spaces with variable order and variable exponents }\label{sec 3}
In this section, we introduce fractional Sobolev spaces with variable order and variable exponents and establish the preliminary lemmas and embeddings 
associated with these spaces.
For that, we assume that $s(\cdot, \cdot)$   and $p(\cdot, \cdot)$ satisfy  $(S1)$ and $(P1)$, respectively. 
We also assume that $\beta\in C_+(\overline{\Om})$.
Recalling the definition of the Lebesgue spaces with variable exponents 
in \cite{Fan1},
we introduce the  fractional Sobolev 
space with variable order and variable exponents as follows:
{\small\begin{align*}
		&W=W^{s(x,y),\beta(x),p(x,y)}(\Om)\\
		&:=\bigg \{ u\in L^{\beta(x)}(\Om):
		\int_{\Om\times\Om}\frac{| u(x)-u(y)|^{p(x,y)}}{\la^{p(x,y)}| x-y |^{N+s(x,y)p(x,y)}}dxdy<\infty,
		\text{ for some }\la>0\bigg\}.
\end{align*}}
We set 
$$[u]_{\Om}^{s(x,y), p(x,y)}:=\DD\inf \left\{\la >0 : \int_{\Om\times\Om}\frac{|
	u(x)-u(y)|^{p(x,y)}}{\la^{p(x,y)}| x-y |^{N+s(x,y)p(x,y)}}dxdy<1\right\} $$
	as semi-norm.
\noi Then $(W, \|\cdot\|_{W})$ is a reflexive Banach space equipped with the norm 
\begin{equation*}
	\| u \|_{W}:= \| u \|_{L^{\beta(x)}(\Om)}+
	[u]_{\Om}^{s(x,y), p(x,y)}.
\end{equation*}
\begin{remark}\label{ineq}
If $A_1\subseteq A_2$ are two bounded open sets in $\RR^N,$ then one can check that $$[u]_{A_1}^{s(x,y)p(x,y)}\leq [u]_{A_2}^{s(x,y)p(x,y)}.$$
\end{remark}
\noi We have the following  Sobolev-type embedding theorem for $W$.
\begin{theorem}\label{thm embd}
	Let $\Omega \subset \mathbb{R}^N$, $N\geq 2$ be a smooth bounded domain and $s(\cdot, \cdot),$ $p(\cdot, \cdot)$ 
	satisfy  $(S1)$ and  $(P1)$, respectively. Let
	$\beta\in C_+(\overline{\Om})$ such that $\beta(x)\geq p(x,x)$ for all $x\in \overline{\Om}.$
	Assume that $\ga\in C_{+}(\overline{\Om})$ 
such that $  \ga(x)<p_s^{*}(x)$
	for $x\in\overline \Omega$. Then
	there exits a constant $K=K(N,s,p,\beta,\ga,\Omega)>0$
	such that
	for every $u\in W$,
	$$
	\| u \|_{L^{\ga(x)}(\Omega)}\leq K \| u \|_{W}.$$			 			
	Moreover, this embedding is compact. 	 
	\end{theorem}
\begin{proof} Here we follow the approach as in \cite{KRV}. Since  $p,~\beta,~\ga,~s$ are continuous on the compact set $\overline{\Om}$, 
	it follows that\begin{align}\label{99}
	\DD\inf_{x\in\overline{\Om}}\Big\{\frac{Np(x,x)}{N-{s}(x,x){p}(x,x)}-\ga(x)\Big\}= k_1 > 0.
	\end{align} Using \eqref{99} and continuity of the functions $p,\beta,\ga$ and $s$, 
	we get a finite  family of disjoint open balls $\{B'_i\}_{i=1}^k$ with radius $\epsilon=\epsilon(p,\beta,\ga,s,k_1)$ 
	satisfying 
	$\overline{\Om}\subseteq\DD\cup_{i=1}^k~B'_i $  such that 
	\begin{align}\label{0.3}
		\frac{Np(z,y)}{N-s(z,y)p(z,y)}-\ga(x)= \frac {k_1}{2} > 0
	\end{align} 
	for all $(z,y)\in B_i\times B_i$ and $x\in B_i$, $i=1,2,\dots,k,$  where $B_i= \Om\cap B'_i ~$ for each $i=1,2,\dots,k.$ \\
	We set
	\begin{align}\label{0.1}
		p_i=\DD\inf_{(z,y)\in B_i\times B_i}(p(z,y)-\delta)
	\end{align}
	and
	\begin{align}\label{0.4}
		s_i=\DD\inf_{(z,y)\in B_i\times B_i}s(z,y). 
	\end{align}
	\noi Again by using continuity of $p,q,\ga$ and $s$ we can  choose $\delta=\delta(k_1)$, with $p^--1>\delta>0$, $t_i\in(0,s_i)$ and $ \epsilon >0$ 
	such that  \eqref{0.3}, \eqref{0.1} and \eqref{0.4} give us
	{\begin{align}\label{0.5}
			{p}_{t_i}^*:=\frac{Np_i}{N-t_ip_i}\geq\frac{k_1}{3}+\ga(x)
	\end{align}}
	and 
	{\begin{align}\label{0.6}
			\beta(x)\geq {p}(x,x)> p_i
	\end{align}}
	for all $x\in B_i,$ $i=1,2,\dots,k.$
	Indeed, from \eqref{0.1} as $ p_i=\DD\inf_{(z,y)\in B_i\times B_i}p(z,y)-\delta(k_1)<{p}(x,x)\leq \beta(x)$ for each $x\in B_i$, we have \eqref{0.6}.  
	Using  embedding results (Theorem 6.7, Theorem 6.9 in \cite{ HG}) for fractional Sobolev spaces, we get a constant  $C=C(N,p_i,t_i,\epsilon,B_i)>0$ such that
	{\begin{align}\label{0.7}
			\parallel u \parallel_{L^{{p}_{t_i}^*}(B_i)}\leq C \bigg\{\| u \|_{L^{p_i}(B_i)}+\bigg( \int_{B_i\times B_i}
			\frac{| u(x)-u(y)|^{p_i}}{| x-y|^{N+{t_i} {p_i}}}dxdy\bigg)^{\frac{1}{p_i}}\bigg\}.
	\end{align}} 
Since $|u(x)|=\DD\sum_{i=1}^k|u(x)|{X}_{B_i}$,
	we have
	{\begin{align}\label{0.8}
			\|u\|_{L^{\ga(x)}(\Om)}\leq \DD\sum_{i=1}^k\|u\|_{L^{\ga(x)}(B_i)}.
	\end{align}}
	From \eqref{0.5}, we get $\ga(x)<{p}_{t_i}^*$ for all $x\in B_i,$ $i=1,...,k$.
	Hence we can take $a_i\in C_+(\Om)$ such that $\frac{1}{\ga(x)}=\frac{1}{{p}_{t_i}^*}+\frac{1}{a_i(x)}.$ 
	Therefore by applying H\"{o}lder's inequality, we obtain
	{\begin{align}\label{0.9}
			\|u\|_{L^{\ga(x)}(B_i)}\leq k_2\|u\|_{L^{{p}_{t_i}^*}(B_i)}~\|1\|_{L^{a_i(x)}(B_i)}
			\leq k_3\|u\|_{L^{{p}_{t_i}^*}(B_i)},	
		\end{align} where the constants $k_2,k_3>0.$}
Hence from \eqref{0.8} and \eqref{0.9}, we deduce
	\begin{align}\label{0.10}
		\|u\|_{L^{\ga(x)}(\Om)}\leq k_4\DD\sum_{i=1}^k\|u\|_{L^{{p}_{t_i}^*}(B_i)},
	\end{align}
	where  $k_4>0$ is a constant.
	Again from \eqref{0.6}, we get $p_i<\beta(x)$ for all $x\in B_i$, $i=1,...,k$. Therefore by arguing in a similar way as above, we obtain  
	\begin{align}\label{0.11}
		\DD\sum_{i=1}^m\|u\|_{L^{p_i}(B_i)}\leq k_5 \|u\|_{L^{\beta(x)}(\Om)},
	\end{align} where $k_5>0$ is a constant.
	Next, for each $i=1,...,k$, we can choose $b_i\in C_+(B_i\times B_i)$ such that
	 \begin{align}\label{110}
	\frac{1}{p_i}=\frac{1}{p(x,y)}+\frac{1}{b_i(x,y)}.
	\end{align}
	We define a  measure in $B_i\times B_i$, as \begin{align}\label{222}
	\DD d\tilde{\mu}(x,y)=\frac{dxdy}{|x-y|^{N+(t_i-s(x,y))p_i}}.\end{align} Using  H\"{o}lder's inequality combining with \eqref{110} and \eqref{222}, 
	it follows that  there exist some constants $k_6,k_7>0$ such that
	{\begin{align}\label{0.12}
			&\bigg\{ \int_{B_i\times B_i}\frac{| u(x)-u(y)|^{p_i}}{| x-y|^{N+{t_i} {p_i}}}dxdy\bigg\}^{\frac{1}{p_i}}\nonumber\\
			&=\bigg\{ \int_{B_i\times B_i}\bigg(\frac{| u(x)-u(y)|}{| x-y|^{s(x,y)}}\bigg)^{p_i} \frac{dxdy}{|x-y|^{N+(t_i-s(x,y))p_i}}\bigg\}^{\frac{1}{p_i}}\nonumber\\
			&= \bigg[ \int_{B_i\times B_i}\big( U(x,y)\big)^{p_i} d\tilde{\mu}(x,y)\bigg]^{\frac{1}{p_i}}\nonumber\\
			&\leq k_6~\|U\|_{L^{p(x,y)}(\tilde{\mu},~ B_i\times B_i)} ~ \|1\|_{L^{b_i(x,y)}(\tilde{\mu},~ B_i\times B_i)}\nonumber\\
			&\leq k_7 ~\|U\|_{L^{p(x,y)}(\tilde{\mu},~ B_i\times B_i)},
	\end{align}}
	where the function $U$ is defined in $B_i\times B_i$ as $U(x,y)=\frac{| u(x)-u(y)|}{| x-y|^{s(x,y)}}, ~ x\neq y.$
	Now let $\la'>0$ be such that 
	{\begin{align}\label{0.13}
			\int_{B_i\times B_i}\frac{|u(x)-u(y)|^{p(x,y)}}{(\la')^{p(x,y)}| x-y|^{N+s(x,y)p(x,y)}}dxdy<1.
	\end{align}}
	Choose
	{\begin{align} \label{0.14}
			d_i:=\DD\sup\bigg\{  1, ~\DD\sup_{(x,y)\in B_i\times B_i}|x-y|^{s(x,y)-t_i}\bigg\}\text{~ and ~ } \overline{\la_i}=\la' d_i.
	\end{align}}
	Combining \eqref{0.13} and \eqref{0.14}, we deduce
	{\begin{align}
			&\int_{B_i\times B_i}\bigg(\frac{U(x,y)}{\overline{\la_i}}\bigg)^{p(x,y)} d\tilde{\mu}(x,y)\nonumber\\
			&=\int_{B_i\times B_i}\bigg(\frac{| u(x)-u(y)|}{\overline{\la_i}| x-y|^{s(x,y)}}\bigg)^{p(x,y)} \frac{dxdy}{|x-y|^{N+(t_i-s(x,y))p_i}}\nonumber\\
			&=\int_{B_i\times B_i}\frac{|x-y|^{(s(x,y)-t_i)p_i}}{{d_i}^{p(x,y)}}\frac{| u(x)-u(y)|^{p(x,y)}}{{(\la')}^{p(x,y)}| x-y|^{N+s(x,y)p(x,y)}}dxdy \nonumber\\
			&\leq\int_{B_i\times B_i} \frac{| u(x)-u(y)|^{p(x,y)}}{{(\la')}^{p(x,y)}| x-y|^{N+s(x,y)p(x,y)}}dxdy <1.\nonumber
	\end{align}}
	Thus from the above, we obtain
	$\|U\|_{L^{p(x,y)}(\tilde{\mu},~ B_i\times B_i)}\leq\overline{\la_i} = \la' d_i,
	$ which together with Remark \ref{ineq} implies
	{\begin{align}\label{0.15}
			\|U\|_{L^{p(x,y)}(\tilde{\mu},~ B_i\times B_i)}\leq k_8 [u]_{B_i}^{s(x,y),p(x,y)}
			\leq k_8 [u]_{\Om}^{s(x,y),p(x,y)},
	\end{align}}
	where $k_8=\DD\max_{\{i=1,2,\dots,k\}}\{d_i\}>1.$
	Taking into account \eqref{0.12} and \eqref{0.15}, we get
	$$\DD\bigg\{\DD \int_{B_i\times B_i}\frac{| u(x)-u(y)|^{p_i}}{| x-y|^{N+{t_i} {p_i}}}dxdy\bigg\}^{\frac{1}{p_i}}\leq k_8 [u]^{s(x,y), p(x,y)}_\Om,$$ which gives us
	{\begin{align}\label{0.16}
			\DD\sum_{i=1}^m \bigg\{\int_{B_i\times B_i}\frac{| u(x)-u(y)|^{p_i}}{| x-y|^{N+{t_i} {p_i}}}dxdy\bigg\}^{\frac{1}{p_i}}\leq k_{9} [u]^{s(x,y), p(x,y)}_\Om,
	\end{align}} where $k_9>0$ is a constant.
	Thus, using \eqref{0.10}, \eqref{0.11} and \eqref{0.16}, we deduce
	{\begin{align*}
			\|u\|_{L^{\ga(x)}(\Om)}&\leq k_4\DD\sum_{i=1}^k\|u\|_{L^{p_i^*}(B_i)}\\
			&\leq k_{10}\DD\sum_{i=1}^m\bigg\{\|u\|_{L^{p_i}(B_i)}+ \bigg(\int_{B_i\times B_i}\frac{| u(x)-u(y)|^{p_i}}{| x-y|^{N+{t_i} {p_i}}}dxdy\bigg)^{\frac{1}{p_i}}\bigg\}\\
			&\leq k_{11} \bigg\{ \|u\|_{L^{\beta(x)}(\Om)} +
			[u]^{s(x,y), p(x,y)}_\Om\bigg\}\\
			&=K(N,s,p,\beta,\ga,\Om) \|u\|_{W},
	\end{align*}} where the constants $k_{10},k_{11}$ and $K>0.$
	This proves that the space $W$ is continuously embedded in
	$L^{\ga(x)}(\Omega)$. The compactness of this embedding in the bounded domain $\Om$ can be established by 
	suitably  extracting a convergent subsequence in $L^{\ga(x)}(B_i)$ for each $i=1,...,k$
	of a bounded sequence 
	$\{u_m\}$ in $W$ and arguing as above.\end{proof}
\noi Next we state the Sobolev-type embedding theorem for $W^{s(x,y),p(x,x),p(x,y)}(\R)$. The proof follows using the similar arguments
as in \cite{ky ho}, 
where the authors have considered the case $s(x,y)=s$, constant.
	\begin{theorem}\label{Sobolev for unbounded}
	Let $s(\cdot, \cdot)$, $p(\cdot, \cdot)$ be uniformly continuous functions
	satisfying  $(S1)$ and  $(P1)$, respectively.
	Assume that $\ga\in C_{+}(\RR^N)$ is a uniformly continuous 
such that $ \ga(x)\geq p(x,x)$ for $x\in\RR^N$ and $\DD\inf_{x\in\RR^N}\left( p_s^{*}(x)-\gamma(x)\right)>0.$
Then $W^{s(x,y), p(x,x), p(x,y)}(\R)$ is continuously embedded into $ L^{\ga(x)}(\RR^N)$.
	 		\end{theorem}
For studying nonlocal problems 
involving the operator $(-\Delta)_{p(\cdot)}^{s(\cdot)}$ with Dirichlet boundary data via variational methods, we define another new fractional type 
Sobolev spaces with variable order and variable exponents. 
One can refer \cite{Bisci} and references therein for this type of spaces in fractional $p$-Laplacian framework.
Let $(S1)$, $(P1)$ hold true and the variable exponent $\beta\in C_+(\overline\Om)$ such that $\beta(x)\geq p(x,x)$ for all $x\in \overline{\Om}$.
Set $Q:=(\RR^N\times\RR^N)\setminus(\Om^c\times\Om^c)$ and define
{\small\begin{align*}
	 {X}&=X^{s(x,y),\beta(x),p(x,y)}(\Om)\\& :=\bigg\{u:\RR^N \rightarrow\mathbb R:u_{|_\Om}\in L^{\beta(x)}(\Om),\\
		&\;\;\;\;\;\;\;\;\int_Q\frac{|u(x)-u(y)|^{p(x,y)}}{\la^{p(x,y)}| x-y|^{N+s(x,y)p(x,y)}}dxdy<\infty,\text{ for some } \la>0\bigg\}.
\end{align*}}
The space ${X}$ is a normed linear space equipped with the following norm:
{\small$$\| u\|_{ {X}}:=\| u\|_{L^{\beta(x)}(\Om)}+\inf\Big\{\la>0:\int_Q\frac{|
		u(x)-u(y)|^{p(x,y)}}{\la^{p(x,y)}| x-y |^{N+s(x,y)p(x,y)}}dxdy<1\Big\}.$$}
Next we define the subspace ${X}_0$ of ${ X}$ as 
$${ X_0}={ X}_0^{s(x,y),\beta(x),p(x,y)}(\Om):=\{u\in { X}\;:\; u=0\;a.e.\; in\;\Om^c\}.$$
It can be verified that the following is a norm on ${ X}_0$, defined as :
\begin{align*}
	\| u\|_{{ X}_0}&:=\inf\Big\{\la>0:\int_Q\frac{|
		u(x)-u(y)|^{p(x,y)}}{\la^{p(x,y)}| x-y |^{N+s(x,y)p(x,y)}}dxdy<1\Big\}.
\end{align*}
\begin{remark}
For $u\in X_0,$ it follows that $$\int_{Q}\frac{|
	u(x)-u(y)|^{p(x,y)}}{\la^{p(x,y)}| x-y |^{N+sp(x,y)}}dxdy=\int_{\RR^N\times\RR^N}\frac{|
	u(x)-u(y)|^{p(x,y)}}{\la^{p(x,y)}| x-y |^{N+sp(x,y)}}dxdy.$$ Hence $$\| u\|_{{ X}_0}=\inf\Big\{\la>0:\int_{\RR^N\times\RR^N}\frac{|
	u(x)-u(y)|^{p(x,y)}}{\la^{p(x,y)}| x-y |^{N+s(x,y)p(x,y)}}dxdy<1\Big\}.$$
\end{remark}
\begin{definition}
For $u\in { X}_0$, we define the modular function $\rho_{{ X}_0}:{ X}_0\ra \RR$:
\begin{equation}\label{modular}
\rho_{{ X}_0}(u):=\int_{\RR^N \times \RR^N}\frac{|
	u(x)-u(y)|^{p(x,y)}}{| x-y |^{N+s(x,y)p(x,y)}}dxdy.\end{equation}
\end{definition}

The interplay between the norm in ${ X}_0$ and the modular function $\rho_{{ X}_0}$ can be studied in the following lemmas:		 	                                                       
\begin{lemma}\label{lem 3.1}
	Let $u \in { X}_0$ and $\rho_{{ X}_0}$ be  defined as in \eqref{modular}. Then we have the following results:
	\begin{enumerate}
		\item[(i)]$ \| u \|_{{ X}_0}<1(=1;>1) \iff \rho_{{ X}_0}(u)<1(=1;>1).$
		\item[(ii)] If $\| u \|_{{ X}_0}>1$, then
		$
		\| u \|_{{ X}_0} ^{p^{-}}\leq\rho_{{ X}_0}(u)\leq\| u \|_{{ X}_0}^{p{+}}.
		$
		\item[(iii)] If $\| u \|_{{ X}_0}<1$, then 
		$\| u \|_{{ X}_0} ^{p^{+}}\leq\rho_{{ X}_0}(u)\leq\| u \|_{{ X}_0}^{p{-}}.
		$
	\end{enumerate}
\end{lemma}		 	
\begin{lemma}\label{lem 3.2}
	Let $u,u_{m}  \in { X}_0$, $m\in\mathbb N$. Then the following statements are equivalent:
	\begin{enumerate}
		\item[(i)] $\DD{\lim_{m\ra \infty} }\| u_{m} - u \|_{{ X}_0} =0,$
		\item[(ii)] $\DD{\lim_{m\ra \infty}} \rho_{{ X}_0}(u_{m} -u)=0.$
	\end{enumerate}
\end{lemma}
\noi The proofs of Lemma \ref{lem 3.1} and Lemma \ref{lem 3.2}  follow in the same line as the proofs of 
Theorem 3.1 and Theorem 3.2, respectively, in \cite{Fan1}. \\
\noi{Now, we study the following}
Sobolev-type embedding theorem for the space ${ X}_0$. The proof of this theorem is motivated from \cite{AB}, where the author studied the result for $s(x,y)=s,$ constant. 
\begin{theorem}\label{prp 3.3}
Let $\Om$ be a smooth bounded domain in $\RR^N, N\geq2$
and $s(\cdot,\cdot)$ and $p(\cdot,\cdot)$ satisfy $(S1)$ and $(P1),$ respectively. Let
	$\beta\in C_+(\overline{\Om})$ such that $ p(x,x)\leq \beta (x)<p^*_s(x)$ for all $x\in \overline{\Om}.$  		
	Then for any  $\ga\in C_+(\overline{\Om})$  with $1<\ga(x)< p_s^*(x)$ for all $x\in \overline{\Om}$,
	there exits a constant $C=C(N,s,p,\ga,\beta,\Omega)>0$
	such that for every $u\in{ X}_0$, 
	\begin{equation*}
	\| u \|_{L^{\ga(x)}(\Omega)}\leq C \| u \|_{{ X}_0}.
	\end{equation*}
	Moreover, this embedding is
	compact.
\end{theorem}
\begin{proof} First we note that, as $\beta$ is continuous on $\overline{\Om}$, using Tietze extension theorem,
	we can extend $\beta$ on  $\mathbb{R}^N$ continuously such that $\beta$ satisfies ${p}(x,x)\leq \beta(x)< p^*_s(x)$  for all $x\in \RR^N.$
	Also $\ga\in C_+(\overline\Om)$ can be extended continuously  on $\RR^N$ such that $1<\gamma(x)< p_s^*(x)$ for all $x\in \RR^N.$				 
	Next, we claim that there exists a constant $C'>0$ such that
	{\begin{align}\label{1.1}
			\|u\|_{L^{\beta(x)}(\Om)}\leq\frac{1}{C'}\|u\|_{{ X}_0} \text{ for all } u\in { X}_0.
	\end{align} }
	This is equivalent to proving that, for $A:=\{u\in{ X}_0:\| u\|_{L^{\beta(x)}(\Om)}=1\}$, $\DD\inf_{u\in A} \| u\|_{{ X}_0}$ is achieved.
	Let $\{u_m\}\subset A$ be a minimizing sequence, that is,  $\| u_m\|_{{ X}_0}\downarrow\inf_{u\in A} \| u\|_{{ X}_0}:=C'$ as $m\ra\infty.$
	This implies that $\{u_m\}$ is bounded in ${ X}_0$ and $L^{\beta(x)}(\Om)$ and hence in $W.$ Therefore, up to a subsequence $u_m\rightharpoonup u_0$ in $W$ as $m\ra\infty.$ 
	Now from Theorem \ref{thm embd}, it follows that $u_m \ra u_0$ strongly in $L^{\beta(x)}(\Om)$ as $m\ra\infty.$
	We extend $u_0$ to $\RR^N$ by setting $u_0(x)=0$ on $x\in \Om ^c.$ This implies $u_m(x) \ra u_0(x)$ a.e. $x\in \RR^N$ as $m\ra\infty.$ Hence by using Fatou's Lemma, we have
	{\small \begin{align*}
			\int_{\RR^N \times \RR^N}\frac{| u_0(x)-u_0(y)|^{p(x,y)}}{| x-y|^{N+s(x,y)p(x,y)}}dxdy \leq 
			\liminf _{m\ra\infty}\int_{\RR^N \times \RR^N}\frac{| u_m(x)-u_m(y)|^{p(x,y)}}{| x-y|^{N+s(x,y)p(x,y)}}dxdy,
	\end{align*}}
	which implies that $\|u_0\|_{{ X}_0}\leq \DD{\lim\inf}_{m\ra\infty}\|u_m\|_{{ X}_0}=C'$ and thus $u_0\in { X}_0$. Also, as
	$\| u_0\|_{L^{\beta(x)}(\Om)}=1$, we get $u_0\in A$. Therefore $\|u_0\|_{{ X}_0}=C'$. 
	This proves our claim and hence \eqref{1.1}.
	From \eqref{1.1}, it follows that 
	{\small \begin{align*}\|u\|_{W}=\|u\|_{L^{\beta(x)}(\Om)}+[u]_{\Om}^{s(x,y), p(x,y)}\leq \|u\|_{L^{\beta(x)}(\Om)}+\|u\|_{{ X}_0}\leq(1+\frac{1}{C'})\|u\|_{{ X}_0},\end{align*}}
	which implies that ${ X}_0$ is continuously embedded in $W$. As Theorem \ref{thm embd} gives that $W$ is continuously embedded in  $L^{\ga(x)}(\Om)$, it follows that there
	exists a constant $C(N,s,p,\ga,\beta,\Om)$ $>0$ such that  
	{ \begin{align}\label{0.0.0}
	\| u \| _{L^{\ga(x)}(\Om)}\leq C(N,s,p,\gamma,\beta,\Om)\| u \| _{{ X}_0}.
	\end{align} }
	To prove that the embedding given in \eqref{0.0.0} is compact, let $\{v_m\}$ be a bounded sequence in ${ X}_0 $.
	This implies that $\{v_m\}$ is bounded in $W$. 
	Hence using Theorem \ref{thm embd}, we infer that there exists $v_0\in L^{\ga(x)}(\Om)$ such that up to a
	subsequence $v_m\ra v_0$ strongly in $L^{\gamma(x)}(\Om)$ as $m\ra\infty$. This completes the theorem.\end{proof}
Using Theorem \ref{prp 3.3} together with the fact that $X_0$ is a closed subspace of
the reflexive space $W^{s(x,y),\beta(x),p(x,y)}(\RR^N)$ with respect to the norm $\|\cdot\|_{X_0}$, we have the following proposition.
\begin{proposition}\label{reflx}
	$({ X}_0, \|\cdot\|_{{ X}_0})$ is a uniformly convex and reflexive Banach space.
\end{proposition}
\begin{remark}\label{solution space}	
	From now onwards we take $\beta(x)=p(x,x)$ and consider the function space\\ $X_0^{s(x,y),p(x,x),p(x,y)}(\Om)$ as the solution space for problem \eqref{mainprob1} 
	and \eqref{mainprob}.	
	For brevity, we still denote the space $X_0^{s(x,y),{p}(x.x),p(x,y)}(\Om)$ by $X_0$.
\end{remark}
\section{Proofs of main theorems}\label{sec 4}
In Proposition 2.4 in \cite{Alves5}, Alves et al. have established the Hardy-Sobolev-Littlewood-type result for variable exponents.
Here we establish the analogous result appropriate to prove \eqref{hls1} for functions in $X_0$.
\begin{proposition}
 Let $s,p,q,\mu$ and $r$ be as in Theorem \ref{HLS1}. Let there exists sequences $\{(x_n,y_n)\}$ and $\{(x_n^\prime,y_n^\prime)\}$ in $\R$ such that
 $\DD\lim_{n\rightarrow\infty}p(x_n,y_n)=p^+$, $\DD\lim_{n\rightarrow\infty}q(x_n,y_n)=q^+$ and
 $\DD\lim_{n\rightarrow\infty}p(x_n^\prime,y_n^\prime)=p^-$, $\DD\lim_{n\rightarrow\infty}q(x_n^\prime,y_n^\prime)=q^-$.
 Then for $h\in L^{p^+}(\RR^N)\cap L^{p^-}(\RR^N)$ and $g \in L^{q^+}(\RR^N)\cap L^{q^-}(\RR^N)$  we have
 {\small\begin{equation*}
  \label{corHLS1}  
 \left|\int_{\RR^N}\int_{\RR^N}\frac{hg}{|x-y|^{\mu(x,y)}}dxdy\right|\leq 
 c_1\left(\|h\|_{L^{p^+}(\RR^N)}\|g\|_{L^{q^+}(\RR^N)}+\|h\|_{L^{p^-}(\RR^N)}\|g\|_{L^{q^-}(\RR^N)}\right),
\end{equation*}}
where $c_1>0$ is a constant, independent of $h$ and $g$. 
\end{proposition}
\begin{proof}
From \eqref{q1}, first we note that 
\begin{align*}
 \mu^+=\sup_{\R}\mu(x,y)&=2N\left(1-\frac{1}{2p(x,y)}-\frac{1}{2q(x,y)}\right)\nonumber\\
 &\leq 2N\left(1-\frac{1}{2p^+}-\frac{1}{2q^+}\right).\nonumber
\end{align*}
Also as 
\begin{align}\DD\lim_{n\rightarrow\infty}\mu(x_n,y_n)&=2N\DD\lim_{n\rightarrow\infty}
\left(1-\frac{1}{2p(x_n,y_n)}-\frac{1}{2q(x_n,y_n)}\right)\nonumber\\
&=2N
\left(1-\frac{1}{2\DD\lim_{n\rightarrow\infty}p(x_n,y_n)}-\frac{1}{2\DD\lim_{n\rightarrow\infty}q(x_n,y_n)}\right)\nonumber\\
&=2N\left(1-\frac{1}{2p^+}-\frac{1}{2q^+}\right),\nonumber
\end{align}
we get that $$\mu^+=2N\left(1-\frac{1}{2p^+}-\frac{1}{2q^+}\right).$$ Similarly we have
$$\mu^-=2N\left(1-\frac{1}{2p^-}-\frac{1}{2q^-}\right).$$ Now as
{\small \begin{align}
 \left|\int_{\RR^N}\int_{\RR^N}\frac{hg}{|x-y|^{\mu(x,y)}}dxdy\right|
 &\leq \int_{\RR^N}\int_{\RR^N}\frac{|hg|}{|x-y|^{\mu(x,y)}}dxdy\nonumber\\
 &\leq \int_{\RR^N}\int_{\RR^N}\frac{|hg|}{|x-y|^{\mu^+}}dxdy+
  \int_{\RR^N}\int_{\RR^N}\frac{|hg|}{|x-y|^{\mu^-}}dxdy\nonumber\\
  &\leq c_1\left(\|h\|_{L^{p^+}(\RR^N)}\|g\|_{L^{q^+}(\RR^N)}+\|h\|_{L^{p^-}(\RR^N)}\|g\|_{L^{q^-}(\RR^N)}\right),\nonumber
\end{align}}
where the last inequalities follows from the Hardy-Littlewood-Sobolev type of inequality for constant exponent case in \cite{Lieb2}. 
\end{proof}
{\bf Proof of Theorem \ref{HLS1}:}\\
Using the Sobolev embedding theorem( Theorem \ref{Sobolev for unbounded}), it is easy to check that 
$|u|^{r(\cdot)}\in L^{q^+}(\RR^N)\cap L^{q^-}(\RR^N)$. Now the proof follows by taking $h(x)=g(x)=|u|^{r(x)}$ and $q(x,y)=p(x,y)$
in the Theorem \ref{HLS1}.\hfill$\square$\\

Next we use the variational method to prove the
existence of solution of \eqref{mainprob}.
We consider the associated energy functional $J:X_0\rightarrow\mathbb R$, given as
{\small \begin{align*}
	J(u)&=\int_{\RR^N \times \RR^N}\frac{|u(x)-u(y)|^{p(x,y)}}{p(x,y)| x-y|^{N+s(x,y)p(x,y)}}dxdy
	-\frac{1}{2}\int_{\Om}\int_{\Om} \frac{F(x,u(x))F(y,u(y)) }{|x-y|^{\mu(x,y)}}dxdy.
\end{align*} }
Note that  as in  section \ref{sec 3} in \cite{Alves5}, Theorem \ref{HLS1} guarantees that $J$ is well-defined and $C^1$ on $X_0$,
with the derivative $J':X_0\ra X_0^*$, given as
\begin{align}\label{4.0}
		\big\langle J'(u),w\big\rangle&= \int_{\RR^N\times \RR^N}\frac{|
			u(x)-u(y)|^{p(x,y)-2}(u(x)-u(y))(w(x)-w(y))}{|
			x-y|^{N+s(x,y)p(x,y)}}dxdy\nonumber\\
		&- \int_{\Om}\int_{\Om} \frac{F(x,u(x))f(y,u(y)) w(y)}{|x-y|^{\mu(x,y)}}dxdy,
\end{align}
where $u,w\in X_0$. Here $\langle\cdot,\cdot\rangle$ denotes the duality pairing between $X_0$ and its dual $X_0^*.$
Thus by the standard critical point theory, the weak solutions of \eqref{mainprob1} are characterized by the 
critical points of $J$.
Also $J$ admits the mountain-pass geometry. Precisely, we have the following lemma.
\begin{lemma}\label{lem 5.1}Let the assumptions in Theorem \ref{existence} hold. Then
	\item[$(i)$] there exists  $\delta>0$, 
	such that $J(u)\geq \zeta>0$ for all $u\in X_0$ with $\| u \|_{X_0}=\delta.$
	\item [$(ii)$] there exists $ \phi\in X_{0}$ with $\| \phi \|_{X_0}>\delta$ such that $J(\phi)<0$.
\end{lemma}
\begin{proof} 
	$(i).$  First using $(F1)$ and Lemma \ref{lem 3.1} and Theorem \ref{prp 3.3}, we note that for $u\in X_0$, $F(x,u(x))\in L^{q^+}(\Om)\cap L^{q^-}(\Om)$.
	Indeed from $(F1)$ we have $F(x,0)=0$ and thus
	{\small \begin{align}\label{sec4eq1}
	 \|F(\cdot,u(\cdot))\|_{L^{q^+}(\RR^N)}=\|F(\cdot,u(\cdot))\|_{L^{q^+}(\Om)}
	 &\leq c_2\left(\int_{\Om}|u(x)|^{r(x)q^+}dx\right)^{1/q^+}\nonumber\\
 	 &\leq c_2\max\left(\|u\|_{L^{r(x)q^+}(\RR^N)}^{r^+},\|u\|_{L^{r(x)q^+}(\RR^N)}^{r^-}\right)\nonumber\\
	 &\leq c_3\max\left(\|u\|_{X_0}^{r^+},\|u\|_{X_0}^{r^-}\right),
	\end{align}}
where the constant $c_2,~ c_3>0$ are independent of $u.$
Similarly for $u\in X_0$ we can check that $F(x,u(x))\in  L^{q^-}(\Om)$. Hence from Theorem \ref{HLS1} and \eqref{sec4eq1} we infer that
{\small \begin{align}
 \label{sec4eq2}
 \left|\int_{\RR^N}\int_{\RR^N}\frac{F(x,u(x))F(y,u(y))}{|x-y|^{\mu(x,y)}}dxdy\right|&
 \leq C\left(\|F(\cdot,u(\cdot))\|^2_{L^{q^+}(\RR^N)}+\|F(\cdot,u(\cdot))\|^2_{L^{q^-}(\RR^N)}\right)\nonumber\\
	 &\leq c_4\left\{\max\{\|u\|_{X_0}^{2r^+},\|u\|_{X_0}^{2r^-}\}\right\}
\end{align}}
for some constants $C,c_4>0$ independent of $u.$
Using Lemma \ref{lem 3.1} and \eqref{sec4eq2}, for $\|u\|_{X_0}<1$,  we have
	\begin{align}
			J(u)&=\int_{\RR^N \times \RR^N}\frac{| u(x)-u(y)|^{p(x,y)}}{p(x,y)| x-y|^{N+s(x,y)p(x,y)}}dxdy\nonumber\\
			&~~~~~~~~-\frac{1}{2}\int_{\Om}\int_{\Om} \frac{F(x,u(x))F(y,u(y)) }{|x-y|^{\mu(x,y)}}dxdy\nonumber\\
			&\geq \frac{1}{p^+} \| u \|_{X_0}^{p^+}-
			c_5\left(\max\{\|u\|_{X_0}^{2r^+},\|u\|_{X_0}^{2r^-}\}\right)\nonumber\\
			&\geq \frac{1}{p^+} \| u \|_{X_0}^{p^+}-
			c_5\|u\|_{X_0}^{2r^-},\nonumber
	\end{align}
	where $c_5>0$ is independent of $u\in X_0$.
	Now noting that $r^->p^+$, we can choose $\delta>0$ sufficiently small such that $J(u)\geq\zeta>0$ for all $u\in X_0$ with $\| u \|_{X_0}=\delta.$\\
	$(ii).\hspace{2mm}$ Recalling Lemma 4 in \cite{SV2} and using $(F2)$, it follows that there  exist two constants $l_1, l_2>0$ such that
	{\begin{align}\label{0.0}
			F(x,t)\geq l_1 |t|^{\Theta/2}
	\end{align}}
	\noi for all $x\in\Om$ and $|t|\geq l_2.$	
	Now, for $\xi\in C^\infty_0(\Om)$ with $\xi>0$ and $t>0$ sufficiently large,
	using Lemma \ref{lem 3.1} and  \eqref{0.0}, we deduce that
	{\small \begin{align*} 
			J(t\xi)
			&=\int_{\RR^N \times \RR^N}\frac{| t\xi(x)-t\xi(y)|^{p(x,y)}}{p(x,y)|x-y|^{N+s(x,y)p(x,y)}}dxdy
			-\frac{1}{2}\int_{\Om}\int_{\Om}\frac{F(x,t\xi(x))F(y,t\xi(y))}{|x-y|^{\mu(x,y)}}dxdy\\
			&\leq \frac{t^{p^{+}}}{p^-}\| \xi \|_{X_0}^{p{^+}}
			-\frac{l_1^2{t^{\Theta}}}{2}\int_{\Om}\int_{\Om}\frac{|\xi(x)|^ {\Theta/2}|\xi(y)|^ {\Theta/2}}{|x-y|^{\mu(x,y)}}dxdy.
	\end{align*}}
	 Since $p^+<\Theta$ in $(F2)$, it follows that $J(t\xi)\ra-\infty$ as $t\ra+\infty.$
	 This guarantees the existence of $ \phi\in X_{0}$ such that $J(\phi)<0$.
\end{proof}
Next we recall Lemma A.1 in \cite{JSG} for variable exponent Lebesgue spaces which is used to prove that $J$ satisfies Palais-Smale condition.
\begin{lemma}\label{lemA1}
	Let $\nu_1(x)\in L^\infty(\Om)$ such that $\nu_1\geq0,\; \nu_1\not\equiv 0.$ Let $\nu_2:\Om\ra\RR$ be a measurable function 
	such that $\nu_1(x)\nu_2(x)\geq 1$ a.e. in $\Om.$ Then for every $u\in L^{\nu_1(x)\nu_2(x)}(\Om),$ 
	$$\parallel |u|^{\nu_1(\cdot)}\parallel_{L^{\nu_2(x)}(\Om)}\leq 
	\parallel u\parallel_{L^{\nu_1(x)\nu_2(x)}(\Om)}^{\nu_1^-}+\parallel u\parallel_{L^{\nu_1(x)\nu_2(x)}(\Om)}^{\nu_1^+}.$$
\end{lemma}	
\begin{lemma}
\label{lem 5.2}
Let the assumptions in Theorem \ref{existence} hold.
Then for any $c\in\RR,$ the functional $J$ satisfies the Palais-Smale $($ in short $(PS)_c)$ condition. 
\end{lemma}
\begin{proof}  
	Let $ \{u_m\} \subset X_0 $ be a $(PS)_c$ sequence of the functional $J$, that is,
	$J(u_m)\ra c$ and $\| J'(u_m)\|_{X_0^*}\ra 0$ as $m\ra \infty$.		 		
	Note that, $\{u_m\}$ is bounded in $X_0$.
	Indeed, if  $\{u_m\}$ is unbounded in $X_0,$ 
for $m$ large enough, using Lemma \ref{lem 3.1} and Theorem \ref{prp 3.3} together with $(F2)$, we have
	{\small\begin{align*}
			&o_m(1)+ c+\|u_m\|_{X_0}\\
			&\geq J(u_m) -\frac{1}{\Theta} \langle J'(u_m),u_m\rangle\\
			&=\int_{\RR^N \times \RR^N}\frac{| u_m(x)-u_m(y)|^{p(x,y)}}{p(x,y)| x-y|^{N+s(x,y)p(x,y)}}dxdy
			-\frac{1}{2}\int_{\Om}\int_{\Om}\frac{F(x,u_m(x))F(y,u_m(y))}{|x-y|^{\mu(x,y)}}dxdy\\
			&-\frac{1}{\Theta}\int_{\RR^N \times \RR^N}\frac{| u_m(x)-u_m(y)|^{p(x,y)}}{| x-y|^{N+s(x,y)p(x,y)}}dxdy
			+\frac{1}{\Theta}{\int_\Om}\int_{\Om}\frac{F(x,u_m(x))f(y,u_m(y))u_m(y)}{|x-y|^{\mu(x,y)}}dxdy\\
			&>\Big(\frac{1}{p^+}- \frac{1}{\Theta}\Big)\int_{\RR^N \times \RR^N}\frac{| u_m(x)-u_m(y)|^{p(x,y)}}{| x-y|^{N+s(x,y)p(x,y)}}dxdy\\
			&+\int_{\Om}\int_{\Om}\frac{F(x,u_m(x))\bigg[\DD\frac{1}{\Theta} f(y,u_m(y))u_m(y)-\frac{1}{2}F(y,u_m(y))\bigg] }{|x-y|^{\mu(x,y)}}dxdy\\
			&> \Big(\frac{1}{p^+}- \frac{1}{\Theta}\Big) \| u_m
			\|_{X_0}^{p^-}.
	\end{align*}} 

	Since $1<p^-\leq p^+<\Theta$  in $(F2)$, from the above expression, we get a contradiction and hence the sequence $\{u_m\} $ is bounded in $X_0.$
	Since $X_0$ is a reflexive Banach space ( Proposition \ref {reflx}), it follows that there exists $u_1\in X_0$ such that up to a subsequence,  
	$u_m \rightharpoonup u_1 $ weakly in $X_0$ and
	$u_m(x)\ra u_1(x)$ point-wise a.e. $x \in \RR^N$ as $m\ra\infty.$ 
We claim that 
	$u_m \rightarrow u_1 $ strongly in $X_0$ as $m\ra\infty.$ We define $I:X_0\ra X_0^*$, as
	\begin{align}\label{5.2} 
		\big\langle I(u), w\big\rangle:=\int_{\RR^N\times \RR^N}\frac{|
			u(x)-u(y)|^{p(x,y)-2}(u(x)-u(y))(w(x)-w(y))}{|
			x-y|^{N+s(x,y)p(x,y)}}dxdy,\nonumber\\
	\end{align}
	where $u, w\in  X_0$.
	Since $\{u_m\}$ is bounded in $X_0$ and  $\| J'(u_m)\|_{X_0^*}\ra 0$  as $ m\ra \infty $, for $w=u_m-u_1$, 
	taking into account \eqref{4.0} and \eqref{5.2}, we deduce  
	{\begin{align}\label{5.3}
			o_m(1)&=\big\langle J'(u_m),(u_m-u_1)\big\rangle \nonumber\\
			&=\langle I(u_m),(u_m-u_1)\rangle -
			\int_\Om\int_{\Om}\frac{F(x,u_m(x) f(y,u_m(y))(u_m-u_1)(y)}{|x-y|^{\mu(x,y)}}dxdy.
	\end{align}}
	Now  using Theorem \ref{HLS1} and Theorem \ref{prp 3.3} we can estimate the second term in 
	the right hand side of \eqref{5.3} as follows.
	First we note that	
	\begin{align}\label{5.4}
 & \left|\int_{\Om}\int_{\Om}\frac{F(x,u_m(x))f(y,u_m(y))(u_m(\cdot)-u_1(\cdot))}{|x-y|^{\mu(x,y)}}dxdy\right|\nonumber\\
&\leq C\|F(\cdot,u_m(\cdot))\|_{L^{q^+}(\Om)}\|f(\cdot,u_m(\cdot))(u_m(\cdot)-u_1(\cdot))\|_{L^{q^+}(\Om)}\nonumber\\
&~~+C\|F(\cdot,u_m(\cdot))\|_{L^{q^-}(\Om)}\|f(\cdot,u_m(\cdot))(u_m(\cdot)-u_1(\cdot))\|_{L^{q^-}(\Om)}\nonumber\\
  &\leq c_6\max\left\{\|u_m\|_{L^{r(x)q^+}(\Om)}^{r^+},\|u_m\|_{L^{r(x)q^+}(\Om)}^{r^-}\right\}\|f(\cdot,u_m(\cdot))(u_m(\cdot)-u_1(\cdot))\|_{L^{q^+}(\Om)}\nonumber\\
  &~~+ c_6\max\left\{\|u\|_{L^{r(x)q^-}(\Om)}^{r^+},\|u\|_{L^{r(x)q^-}(\Om)}^{r^-}\right\}\|f(\cdot,u_m(\cdot))(u_m(\cdot)-u_1(\cdot))\|_{L^{q^-}(\Om)}\nonumber\\
   &\leq c_7\max\left\{\|u_m\|_{X_0}^{r^+},\|u_m\|_{X_0}^{r^-}\right\}\|f(\cdot,u_m(\cdot))(u_m(\cdot)-u_1(\cdot))\|_{L^{q^+}(\Om)}\nonumber\\
  &~~+ c_7\max\left\{\|u_m\|_{X_0}^{r^+},\|u_m\|_{X_0}^{r^-}\right\}\|f(\cdot,u_m(\cdot))(u_m(\cdot)-u_1(\cdot))\|_{L^{q^-}(\Om)},\nonumber\\
 \end{align}
 where the constants $C,c_6,c_7>0$ are independent of $u_m.$
 Now  using $(F1)$, together with H\"{o}lder's inequality and Lemma \ref{lemA1} and the fact that $u_m\ra u_1$ strongly in 
	$L^{q^-r(x)}(\Om)$ as $m\ra\infty$,    we have
 \begin{align}\label{sec4eq3}
 &\|f(\cdot,u_m(\cdot))(u_m(\cdot)-u_1(\cdot))\|^{q^+}_{L^{q^+}(\Om)}\nonumber\\
 &=\int_\Om |f(y,u_m(y))(u_m(y)-u_1(y))|^{q^+}dy\nonumber\\
 &\leq M^{q^+} \int_\Om |u_m(y)|^{(r(y)-1)q^+}|~|(u_m(y)-u_1(y))|^{q^+}dy\nonumber\\
 &\leq c_8\|u_m^{(r(\cdot)-1)q^+}\|_{L^{\frac{r(x)}{r(x)-1}}(\Om)}\|(u_m-u_1)^{q^+}\|_{L^{r(x)}(\Om)}\nonumber\\
 &\leq c_9\left(\|u_m\|^{(r^+-1)q^+}_{L^{r(x)q^+}(\Om)}+\|u_m\|^{(r^--1)q^+}_{L^{r(x)q^+}(\Om)}\right) 
 \|(u_m-u_1)\|_{L^{q^+r(x)}(\Om)}^{q^+}\nonumber\\
 &\leq c_{10}\left(\|u_m\|^{(r^+-1)q^+}_{X_0}+\|u_m\|^{(r^--1)q^+}_{X_0}\right) 
  \|(u_m-u_1)\|_{L^{q^+r(x)}(\Om)}^{q^+}\nonumber\\
 &\leq c_{11}\|(u_m-u_1)\|_{L^{q^+r(x)}(\Om)}^{q^+}=o_m(1).
 \end{align}Here $c_8,c_9,c_{10}$ and $c_{11}$ are non-negative constants independent of $u_m,u_1.$ 
 		Again  arguing similarly as above, we obtain	
	\begin{align}\label{5.4.1}
		\|f(\cdot,u_m(\cdot))(u_m(\cdot)-u_1(\cdot))\|_{L^{q^-}(\Om)}=o_m(1).
	\end{align}	  
	\noi Thus from \eqref{5.3}-\eqref{5.4.1}, we deduce that 
	$$\langle I(u_m),(u_m-u_1)\rangle  \ra 0
	\text{ and }\langle I(u_1),(u_m-u_1)\rangle  \ra 0\text{ as }m\ra\infty,$$
	which imply
	\begin{equation}\label{5.5}
	\big\langle (I(u_m)-I(u_1)),( u_m-u_1)\big\rangle=o_m(1).
	\end{equation}
	Next we recall the following inequalities due to Simon \cite{Simon}: For all $x,y\in \RR^N$, we have
	{\small \begin{equation}\label{5.0}
		\left\{ \begin{array}{rl}
		|x-y|^p &\leq \frac{1}{p-1}\Big[ \left(|x|^{p-2}x-|y|^{p-2}y\right).\left(x-y\right)\Big]^{\frac{p}{2}}
		\left(|x|^p+|y|^p\right)^{\frac{2-p}{2}} ,1<p<2,\\
		|x-y|^p	&\leq {2^p}\Big(|x|^{p-2}x-|y|^{p-2}y\Big).\left(x-y\right)~,~~~~~~~~ p\geq 2,
		\end{array}
		\right.
	\end{equation}}
	
	\noi We define $ \Om_1:=\{(x,y)\in \RR^N\times \RR^N: 1<p(x,y)<2\}$, $\Om_2:=\{(x,y)\in \RR^N\times \RR^N: 
	p(x,y)\geq 2\}$ and denote $v_m=u_m-u_1$.
	Then we have the following estimate.	
	\begin{align}\label{111}	
\ro_{X_0}(v_m)&=\int_{\RR^N\times \RR^N}\frac{| v_m(x)-v_m(y)|^{p(x,y)}}{|
x-y|^{N+s(x,y)p(x,y)}}dxdy \nonumber\\&=\int_{\Om_1}\frac{|
v_m(x)-v_m(y)|^{p(x,y)}}{| x-y|^{N+s(x,y)p(x,y)}}dxdy	 +\int_{\Om_2}\frac{|
				v_m(x)-v_m(y)|^{p(x,y)}}{|
				x-y|^{N+s(x,y)p(x,y)}}dxdy.
				\end{align}
	Now for $(x,y)\in \Om_1,$ taking into account
	Lemma \ref{lem 3.1}, H\"{o}lder's inequality, Lemma \ref{lemA1} and \eqref {5.0}, we deduce
	{\small\begin{align}\label{5.6}
			&\int_{\Om_1}\frac{|
				v_m(x)-v_m(y)|^{p(x,y)}}{|
				x-y|^{N+s(x,y)p(x,y)}}dxdy\nonumber\\
					&\leq \frac{1}{(p^{-}-1)}
							\int_{\Om_1} \bigg[\bigg\{\frac{|
								u_m(x)-u_m(y)|^{p(x,y)-2}(u_m(x)-u_m(y))(v_m(x)-v_m(y))}{|
								x-y|^{N+s(x,y)p(x,y)}}\nonumber\\
							&-\frac{|
								u_1(x)-u_1(y)|^{p(x,y)-2}(u_1(x)-u_1(y))(v_m(x)-v_m(y))}{|
								x-y|^{N+s(x,y)p(x,y)}}\bigg\}^{p(x,y)/2}\nonumber\\
							&\times \bigg\{\frac{|
								u_m(x)-u_m(y)|^{p(x,y)}+|
								u_1(x)-u_1(y)|^{p(x,y)}}{|
								x-y|^{N+s(x,y)p(x,y)}}\bigg\}^{\frac{2-p(x,y)}{2}}\bigg] dxdy\nonumber\\
			&\leq \frac{1}{(p^{-}-1)}
			\int_{\RR^N\times \RR^N} \bigg[\bigg\{\frac{|
				u_m(x)-u_m(y)|^{p(x,y)-2}(u_m(x)-u_m(y))(v_m(x)-v_m(y))}{|
				x-y|^{N+s(x,y)p(x,y)}}\nonumber\\
			&-\frac{|
				u_1(x)-u_1(y)|^{p(x,y)-2}(u_1(x)-u_1(y))(v_m(x)-v_m(y))}{|
				x-y|^{N+s(x,y)p(x,y)}}\bigg\}^{p(x,y)/2}\nonumber\\
			&\times \bigg\{\frac{|
				u_m(x)-u_m(y)|^{p(x,y)}+|
				u_1(x)-u_1(y)|^{p(x,y)}}{|
				x-y|^{N+s(x,y)p(x,y)}}\bigg\}^{\frac{2-p(x,y)}{2}}\bigg] dxdy\nonumber\\
			&\leq \frac{1}{(p^{-}-1)}
			\int_{\RR^N\times \RR^N}\bigg[\bigg\{\frac{|
				u_m(x)-u_m(y)|^{p(x,y)-2}(u_m(x)-u_m(y))(v_m(x)-v_m(y))}{|
				x-y|^{N+s(x,y)p(x,y)}}\nonumber\\		
			&-\frac{|
				u_1(x)-u_1(y)|^{p(x,y)-2}(u_1(x)-u_1(y))(v_m(x)-v_m(y))}{|
				x-y|^{N+s(x,y)p(x,y)}}
			\bigg\}^{p(x,y)/2}\nonumber\\		
			&\times\bigg \{\bigg(\frac{|
				u_m(x)-u_m(y)|^{p(x,y)}} {|
				x-y|^{N+s(x,y)p(x,y)}}\bigg)^{\frac{2-p(x,y)}{2}}
			+\bigg(\frac{|
				u_1(x)-u_1(y)|^{p(x,y)}} {|
				x-y|^{N+s(x,y)p(x,y)}}\bigg)\bigg\}^{\frac{2-p(x,y)}{2}}  \bigg] dxdy\nonumber\\
			&= c_{12} \int_{\RR^N\times \RR^N}\bigg[\bigg\{ g_1(x,y)^{\frac{p(x,y)}{2}}\cdot g_2(x,y)^{\frac{2-p(x,y)}{2}}\bigg\}\nonumber\\
			&+\bigg\{ g_1(x,y)^{\frac{p(x,y)}{2}}\cdot g_3(x,y)^{\frac{2-p(x,y)}{2}}\bigg\}\bigg] dxdy\nonumber\\
			&\leq  c_{12} \bigg[ \| g_1^{\frac{p(\cdot,\cdot)}{2}}\|_{L^{\frac{2}{p(x,y)}}(\RR^N\times\RR^N)}  
			\bigg\{\| g_2^{\frac{2-p(\cdot,\cdot)}{2}}\|_{L^{\frac{2}{2-p(x,y)}}(\RR^N\times\RR^N)}\nonumber\\
			&~~~~~~~~~~~~~+\|g_3^{\frac{2-p(\cdot,\cdot)}{2}}\|_{L^{\frac{2}{2-p(x,y)}}(\RR^N\times\RR^N)}\bigg\}\bigg]\nonumber \\
			&\leq c_{12}\bigg[\bigg\{ \| g_1\|_{L^1(\RR^N\times\RR^N)}^{\frac{p^+}{2}} + \| g_1\|_{L^1(\RR^N\times\RR^N)}^{\frac{p^-}{2}}\bigg \}\nonumber\\
			&\times \bigg\{ \| g_2\|_{L^1(\RR^N\times\RR^N)}^{\frac{2-p^+}{2}} + \| g_2\|_{L^1(\RR^N\times\RR^N)}^{\frac{2-p^-}{2}}
			+\|g_3\|_{L^1(\RR^N\times\RR^N)}^{\frac{2-p^+}{2}} + \| g_3\|_{L^1(\RR^N\times\RR^N)}^{\frac{2-p^-}{2}} \bigg\} 
			\bigg],\nonumber\\
	\end{align}}
	\noi where $c_{12}>0$ is some constant independent of $u_m,u_1$ and $g_{i},i=1,2,3$ are defined as follows.
	{\small\begin{align*}\DD g_1(x,y)&= \bigg[\frac{|
				u_m(x)-u_m(y)|^{p(x,y)-2}(u_m(x)-u_m(y))(v_m(x)-v_m(y))}{|
				x-y|^{N+s(x,y)p(x,y)}}\\
			&~~-\frac{|
				u_1(x)-u_1(y)|^{p(x,y)-2}(u_1(x)-u_1(y))(v_m(x)-v_m(y))}{|
				x-y|^{N+s(x,y)p(x,y)}}
			\bigg],\\
			\DD g_2(x,y)&= \bigg[\frac{|
		u_m(x)-u_m(y)|^{p(x,y)}} {|
 		x-y|^{N+s(x,y)p(x,y)}}\bigg],\\
		g_3(x,y)&= \bigg[ \frac{|
		u_1(x)-u_1(y)|^{p(x,y)}} {|
		x-y|^{N+s(x,y)p(x,y)}}\bigg].
	\end{align*}}
	Finally using Lemma \ref{lem 3.1} and \eqref{5.6}, it follows that
	{\small\begin{align}\label{5.7}
			&\int_{\Om_1}\frac{|
							v_m(x)-v_m(y)|^{p(x,y)}}{|
							x-y|^{N+s(x,y)p(x,y)}}dxdy\nonumber\\&\leq c_{12} \bigg[ \bigg\{\big\langle (I(u_m)-I(u_1)),(u_m-u_1)\big \rangle^{\frac {p^+}{2}} 
			+ \big\langle(I(u_m)-I(u_1)),(u_m-u_1)\big\rangle^{\frac {p^-}{2}}\bigg\}\nonumber\\
			&~~~~~~~~~~~~~~~\times \bigg\{ \rho_{X_0} (u_m)^{\frac{2-p^+}{2}} 
			+ \rho_{X_0}(u_m)^{\frac{2-p^-}{2}} +\rho_{X_0}(u_1)^{\frac{2-p^+}{2} }
			+ \rho_{X_0}(u_1)^{\frac{2-p^+}{2}}\bigg\}
			\bigg]. 
	\end{align}}
Combining \eqref {5.5} and \eqref{5.7} and using the fact that $\{u_m\}$ is bounded in $X_0,$ we estimate 
{\small \begin{align} \label{333}
\int_{\Om_1}\frac{|
				v_m(x)-v_m(y)|^{p(x,y)}}{|
				x-y|^{N+s(x,y)p(x,y)}}dxdy\ra 0  ~~as~~ m\ra\infty.
				\end{align}}
Again for $(x,y)\in \Om_2$, taking into account  Lemma \ref{lem 3.1}, H\"{o}lder's inequality, Lemma \ref{lemA1}, \eqref {5.5}
	and Simon's inequality \eqref{5.0}, we deduce
{\small \begin{align}\label{444}
\int_{\Om_2}\frac{|v_m(x)-v_m(y)|^{p(x,y)}}{|x-y|^{N+s(x,y)p(x,y)}}dxdy&\leq 2^{p^+}\big\langle(I(u_m)-I(u_1)), (u_m-u_1)\big\rangle
		=o_m(1).
	\end{align}}
	Thus from \eqref{111},  \eqref{333} and \eqref{444}, we get
	$$ \ro_{X_0}(v_m)\ra 0 ~~as~~ m\ra 0.$$
	Therefore Lemma \ref{lem 3.2} gives us $\DD\lim_{m\ra\infty}\| u_m-u_1\|_{X_0}\ra 0$.
	This completes the lemma.
\end{proof}

\noi Now, we give the proof of Theorem \ref{existence}.

\noi\textbf{{ Proof of Theorem \ref{existence}:}}\\
From Lemma \ref{lem 5.1} and Lemma \ref{lem 5.2}, it follows that  $J$
satisfies the Mountain pass geometry and Palais-Smale condition. Therefore by using Mountain pass theorem, we infer that
there exists $u_1\in X_0$,  a critical point of 
$J$, with
\begin{align} \label{123}
J(u_1)=\bar c>0.
\end{align}
Also as $J(0)=0$, thanks to $(F1)$,
 we get that $u_1$ is a non-trivial weak solution of the problem $(\ref{mainprob1})$.

\noi{\bf Proof of Theorem \ref{multiplicity}:}\\
Note that the weak solutions for the problem \eqref{mainprob} are characterized by the critical points of the following 
$C^1$-functional associated with \eqref{mainprob}.
\begin{align*}
	J_\la(u)&=\int_{\RR^N \times \RR^N}\frac{| u(x)-u(y)|^{p(x,y)}}{p(x,y)| x-y|^{N+s(x,y)p(x,y)}}dxdy-\la \int_{\Om}\frac{| u |^ {\al(x)}}{\al(x)}dx\nonumber\\
	&~~~~~~-\frac{1}{2}\int_{\Om}\int_{\Om} \frac{F(x,u(x))F(y,u(y)) }{|x-y|^{\mu(x,y)}}dxdy.
\end{align*}
First we claim that
there exists $\Lambda>0$ such that for every $\la\in(0,\Lambda)$ we can find $\zeta_\la>0$ and $0<\delta_\la<<1$
such that $J_\la(u)\geq \zeta_\la>0$ for all $u\in X_0$ with $\| u \|_{X_0}=\delta_\la$.
Indeed, using Lemma \ref{lem 3.1}, Theorem \ref{prp 3.3}, Theorem \ref{hls1} together with $(F1)$ and $(F2)$, for $\|u\|_{X_0}<1$,  we have
	{\begin{align}\label{NEW2}
			J_\la(u)&=\int_{\RR^N \times \RR^N}\frac{| u(x)-u(y)|^{p(x,y)}}{p(x,y)| x-y|^{N+s(x,y)p(x,y)}}dxdy
			-\la \int_{\Om}\frac{| u |^ {\al(x)}}{\al(x)}dx\nonumber\\
			&~~~~~-\frac{1}{2}\int_{\Om}\int_{\Om} \frac{F(x,u(x))F(y,u(y)) }{|x-y|^{\mu(x,y)}}dxdy\nonumber\\
			&\geq \frac{1}{p^+} \| u \|_{X_0}^{p^+}-\frac{\la}{\al^- } \max\left\{\| u \|_{L^{\al(x)}(\Om)}^{\al^- },
			\| u \|_{L^{\al(x)}(\Om)}^{\al^+ }\right\}\nonumber\\
			&~~~~-c_{13}\max\left\{\| u \|_{X_0}^{2r^- },\| u \|_{X_0}^{2r^+ }\right\}\nonumber\\
			&\geq\frac{1}{p^+} \| u \|_{X_0}^{p^+}-\frac{\la c_{14}}{\al^- } \| u \|_{X_0}^{\al^-  }-c_{13}\| u \|_{X_0}^{2r^- }\nonumber\\
			&\geq\bigg\{\frac{1}{p^+} - \frac{\la c_{14}}{\al^- }\| u \|_{X_0}^{\al^- -p^+ }-{c_{13}}\| u \|_{X_0}^{2r^- -p^+  }
			\bigg\}\| u \|_{X_0}^{p^+},
	\end{align}}
where the constants  $c_{13}, c_{14}>0$ are independent of $u$.	
	Now for each $\la>0$ we define the function,
	$T_{\la}:(0,+\infty)\ra\RR$ as
	{\begin{align*}T_{\la}(t)=c_{14}\frac{\la}{\al^- } t^{\al^- -p^+ }+c_{13}t^{2r^- -p^+  }.\end{align*}}
	 		Since   we have  $1<\al^-<p^+<r^-$, it follows that
	 		$\DD \lim _{t\ra0} T_{\la}(t)=\DD\lim _{t\ra \infty} T_{\la}(t)=+\infty.$ Thus we can find infimum of $ T_\la $.	 		
	 		Note that equating
	 		$$T'_{\la}(t)= \frac{\al^- -p^+ }{\al^- }\la c_{14} +c_{13} (2r^--p^+) t^{2{r}^--\al^-}=0,$$
	 		we get $t_0=t=\bigg( \la \frac{p^+-\al^- }{(2r^--p^+)\al^- }\cdot\frac{c_{14}}{c_{13}}\bigg)^{1/(2r^--\al^-)}.$
	 		Clearly $t_0>0$. Also it can be checked that $T''_{\la}(t_0)>0$ and hence infimum of $T_{\la}(t)$ is achieved at $t_0$.
	 		Now observing that
	 		\begin{align}\label{New2} 
	 		T_{\la}(t_0)&=\la \frac{c_{14}}{\al^-}
	 		\bigg( \la \frac{p^+-\al^-}{(2r^--p^+)\al^-}\cdot\frac{c_{14}}{c_{13}}\bigg)^{\DD\frac{\al^--p^+}{2r^--\al^-}}\nonumber\\
	 		&~~~~~+{c_{13}}
	 		\bigg( \la \frac{p^+-\al^-}{(2r^--p^+)\al^-}\cdot\frac{c_{14}}{c_{13}}\bigg)^{\DD\frac{2r^--p^+}{r^--\al^-}}\nonumber\\
	 		&=\la^{\DD{\frac{2r^--p^+}{r^--\al^-}}}\cdot c_{15}\ra0\text{ as\hspace{2mm}} \la\ra0^+,
	 		\end{align}	
	for some constant $c_{15}>0$ independent of $u.$ 	
	Therefore we infer from \eqref{NEW2} that there exists $\Lambda>0 $ such that for any $\la\in(0,\Lambda),$
	we can choose $\zeta_\la>0$ and $0<\delta_\la<<1$ such that
	\begin{align}\label{New3}
\text{	
 $J_\la(u)\geq \zeta_\la>0 $ for all $ u\in X_0$ with $\| u \|_{X_0}=\delta_\la$.}
	\end{align}
On the other hand, for $\xi\in C_0^\infty(\Om), \xi>0$ we have
	{\begin{align*} 
			J_\la(t\xi)
			&=\int_{\RR^N \times \RR^N}\frac{| t\xi(x)-t\xi(y)|^{p(x,y)}}{p(x,y)| x-y|^{N+s(x,y)p(x,y)}}dxdy
			-\la \int_{\Om}\frac{| t\xi |^ {\al(x)}}{\al(x)}dx\\
			&-\frac{1}{2}\int_{\Om}\int_{\Om}\frac{F(x,t\xi(x))F(y,t\xi(y))}{|x-y|^{\mu(x,y)}}dxdy\\
			&\leq \frac{t^{p^{+}}}{p^-}\| \xi \|_{X_0}^{p{^+}}
			-\frac{l_1^2{t^{\Theta}}}{2}\int_{\Om}\int_{\Om}\frac{|\xi(x)|^ {\Theta/2}|\xi(y)|^ {\Theta/2}}{|x-y|^{\mu(x,y)}}dxdy\\
			&\ra - \infty\text{ as }t\ra+\infty.
	\end{align*}}
	Thus there exists $\Phi\in X_0$ such that
	\begin{equation}\label{New4}
\|\Phi\|_{X_0}>\delta_{\la}\text{ and }J_{\la}(\Phi)<0.
	\end{equation}
	Using the fact that $J_{\la}(0)=0$ (thanks to $(F1)$) and from \eqref{New3}-\eqref{New4}, it follows that
	$J_\la$ admits a mountain-pass geometry.	
	Also using the similar arguments as in the proof of Lemma  \ref{lem 5.2} and the compactness of the embedding $X_0$ into $ L^{\al(x)}(\Om)$,
	one can check that the functional $J_\la$ satisfies the Palais-Smale condition
	$(PS)_c$ for any $c\in\RR$. Hence from Mountain pass theorem, we infer that there exists
a non-trivial weak solution, $u_{1\la}$ (say) of the problem $(\ref{mainprob})$ with $J_\la(u_{1\la})>0$.\\
Now we prove the existence of the second weak solution of \eqref{mainprob}. 
Under the assumptions in Theorem \ref{multiplicity},
there exists $ \varphi \in X_{0},\; \varphi >0$ such that $J_\la(t \varphi) <0$  for all $ t\ra0^+$. Indeed 
for $\varphi\in C^\infty_0(\Om)$ such that $\varphi>0$ and for sufficiently small $0<t$, $\| t\varphi \|_{X_0}<1$.
	Then from $(F2)$ and Lemma \ref{lem 3.1}, we obtain
	\begin{align*} 
		J_\la(t\varphi)&=\int_{\RR^N \times \RR^N}\frac{| t\varphi(x)-t\varphi(y)|^{p(x,y)}}{p(x,y)| x-y|^{N+s(x,y)p(x,y)}}dxdy
		-\la \int_{\Om}\frac{| t\varphi |^ {\al(x)}}{\al(x)}dx\\
		&~~~-\frac{1}{2}\int_{\Om}\int_\Om\frac{F(x,t\varphi) F(y,t\varphi)}{|x-y|^{\mu(x,y)}}dxdy\\
		&\leq\frac{t^{p^{-}}}{p^-}\| \varphi \|_{X_0}^{p{^-}}
		-\frac{\la t^{\al^+}}{\al^+}\int_{\Om}| \varphi(x) |^ {\al(x)}dx.
	\end{align*}
	As  $\al^+<p^-$, it follows that $J(t\varphi)<0$ as $t\ra0^+.$ Hence we infer that
\begin{align}\label{6.0}
 \DD\inf _{u\in \overline {B}_{\delta_\la}(0)} J_\la(u)=\underline{c}<0, 
\end{align} 
where $\overline{B}_{\delta_\la}(0)=\{u\in X_0 : \|u\|_{X_0}\leq\delta_\la\}.$
  Now by applying Ekeland variational principle, for
given any $\epsilon>0$ there exists $w_\epsilon\in\overline {B}_{\delta_\la} (0)$ such that 
\begin{align}\label{6.1}
J_\la(w_\epsilon)&<\DD\inf _{u\in \overline {B}_{\delta_\la}(0)} J_\la(u)+\epsilon
\end{align}
           and
\begin{align}\label{6.2}
J_\la(w_\epsilon)&<J_\la(u)+\epsilon\|u-w_\epsilon\|_{X_0},\;\text{for all}\; u\in B_{\delta_\la}(0),\; u\not=w_\epsilon.
\end{align}
We choose $\varrho>0$ such that 
\begin{align}\label{6.3}
0<\varrho<\DD\inf _{u\in \partial {B}_{\delta_\la}(0)} J_\la(u)-\DD\inf _{u\in \overline {B}_{\delta_\la}(0)} J_\la(u).
\end{align}
Putting together \eqref{6.1} and \eqref{6.3}, we obtain $J_\la(w_\epsilon)< \DD\inf _{u\in \partial {B}_{\delta_\la}(0)} J_\la(u)$, 
which implies $w_\epsilon \in B_\rho(0).$ By  taking $u=w_\epsilon+tv $ in \eqref{6.2} with $t>0$ and $v\in  {B}_{\delta_\la}(0) \setminus\{0\}$, we deduce
\begin{align*}
J_\la(w_\epsilon)-J_\la(w_\epsilon+tv)\leq \delta_\la t\|v\|_{X_0}.
\end{align*}
Thus
\begin{align*}
\DD\lim_{t\ra 0}\frac{J_\la(w_\epsilon)-J_\la(w_\epsilon+tv)}{t}\leq \delta_\la\|v\|_{X_0},
\end{align*}
that is, for all $v\in B_{\delta_\la}(0),$ we have
\begin{align}\label{6.4}
\langle-J'_\la(w_\epsilon),v\rangle\leq \delta_\la\|v\|_{X_0}.
\end{align}
Replacing $v$ by $-v$ in (\ref{6.4}), we get 
\begin{align}\label{6.5}
(J'_\la(w_\epsilon),v)\leq \delta_\la\|v\|_{X_0}.
\end{align}
Taking into account \eqref{6.4} and \eqref{6.5}, we obtain
\begin{align}\label{6.6}
\|J'_\la(w_\epsilon)\|_{X_0^*}\leq \delta_\la.
\end{align}
From \eqref{6.6}, it follows that there exists a sequence $\{w_m\}\subset B_{\delta_\la}(0)$ such that 
$$J_\la(w_m)\ra\underline{c}\;\text{ and }
\; J'_\la(w_m)\ra0 \;\text{ in }\; X_0^*\;\text{ when }\; m\ra\infty.$$
Therefore from Lemma \ref{lem 5.2} and \eqref{6.0}, we can conclude that there exists $u_{2\la}\in \overline{B}_{\delta_\la}(0)\subset X_0$ 
such that $w_m\ra u_{2\la}$ strongly in $X_0$-norm as $m\ra\infty$
with
\begin{align}\label{6.7}
J_\la(u_{2\la})=\underline{c}<0 .
\end{align}
Thus we get that
$u_{2\la} $ is a nontrivial weak solution of (\ref{mainprob}).
Now from \eqref{123} and \eqref{6.7}, 
we have $J_\la(u_{1\la})>0>J_\la(u_{2\la})$, and hence $ u_{1\la} \not= u_{2\la}.$
\hfill$\square$

\end{document}